\documentclass[12pt]{amsart}
\usepackage{amssymb,pb-diagram}
\usepackage{tikz}
\usetikzlibrary{cd}
\usetikzlibrary{matrix,arrows,decorations.pathmorphing}

\title{$K$-bad spheres}

\author{Martin Bendersky}
\author{Robert Thompson}
\address{Department of Mathematics and Statistics, Hunter College and the Graduate Center, CUNY, 
New York, NY 10065}
\email{mbendersk1@gmail.com}
\email{robert.thompson@hunter.cuny.edu}

\numberwithin{equation}{section}
\newtheorem{theorem}[equation]{Theorem}
\newtheorem{definition}[equation]{Definition}

\newtheorem{lemma}[equation]{Lemma}

\DeclareMathOperator{\holim}{holim}

\DeclareMathOperator{\Tor}{Tor}

\newcommand{\Z}{\mathbf{Z}}

\DeclareMathOperator{\Cotor}{Cotor}

\DeclareMathOperator{\Tot}{Tot}

\DeclareMathOperator{\Ho}{Ho}
\DeclareMathOperator{\ch}{\mathit{Ch}_{*}}
\newcommand\thetapring{\mathop{\mbox{$\theta^{p}$-$\mathit{Rng}_p$}}}
\newcommand\thetapringAd{     \mathop{    \mbox{ $\theta^{p}$-$\mathit{Rng}_p$-$\mathit{Ad}_p$}      }       }

\newcommand\abp{\mathcal{A}b_{p} }
\newcommand\babp{\bar{\mathcal{A}}b_{p}}
\newcommand\Adp{Ad_{p} }

\def\S{\mathcal{S}_{*}}
\def\Sp{\mathcal{S}{p}}

\def\HoS{\text{Ho}_*}
\def\HoSp{\text{Ho}^{s}}

\begin{document}

\begin{abstract}
In this paper we look at the $E$-completion of topological spaces where $E$ is a $p$-local ring spectrum.  After a brief review of the concept of $E$-completion, we specialize to the case where $E=K$, $p$-local complex periodic $K$-theory, and consider the $K$-theory of the unstable sphere $S^{2n+1}$. We show that for certain values of $n$ and an odd prime $p$, the $K$-homology of the $K$-completion is not isomorphic to the $K$-homology of the sphere itself, thus in the terminology of Bousfield and Kan, these spheres are '$K$-bad'. 
\end{abstract}

\maketitle

%
%

\section{\bf $E$-completion, $E$-good and $E$-bad spaces}\label{E-completion}

Assume throughout that everything is localized at a given odd prime $p$.  We define $q=2(p-1)$.  The $p$-adic integers are denoted by $\Z_p$. 
We use $\S$ to denote the category of pointed spaces, and $\HoS = \Ho \S$ for the pointed homomotopy category.  
We use $\Sp$ to denote the category of spectra as in \cite{HSS} for example, and $\HoSp = \Ho \Sp$ for the stable homotopy category.  If $A\in \HoS$ 
or $\HoSp$ and $E \in \HoSp$ we use $A_{E}$ to denote the Bousfield localization of $A$ with respect to $E$ (see \cite{BO1},\cite{BO2}).

This paper is based on the framework of 'Cosimplicial resolutions and homotopy spectral sequences in model categories' by A. K. Bousfield~\cite{BO9}.   
Briefly, in that paper Bousfield considers a left proper pointed simplicial model category $\mathcal{C}$ with a chosen class of groups objects 
$\mathcal{G}$ in $\Ho\mathcal{C}$.  The class  $\mathcal{G}$ is used to define a simplicial model category structure on $c\mathcal{C}$, the category of cosimplicial objects in $\mathcal{C}$.  The resulting category, called the 'resolution model category', is denoted by $c\mathcal{C}^{\mathcal{G}}$. If $A$ is an object in 
$\mathcal{C}$,  regarding  $A$ as a constant cosimplicial object,  a $\mathcal{G}$-resolution of $A$ is a trivial cofibration to a fibrant object $A\to Y^{\bullet}$ in $c\mathcal{C}^{\mathcal{G}}$, guaranteed to exist by virtue of the  model category structure.  Then  $\hat{L}_{\mathcal{G}}A$ is defined to be the total object $\Tot Y^{\bullet}$.    
This defines an endofunctor on $\Ho\mathcal{C}$ 
called the $\mathcal{G}$-completion of $A$.   We are primarily interested in the cases where $\mathcal{C} = \S$ or $\Sp$.

In the case of $\mathcal{C} = \S$, if $E$ is a $p$-local ring spectrum, and $\mathcal{G}$ is taken to be the class $\mathcal{G} = \{\Omega^{\infty}N \vert\,\,\text{ $N$ 
is an $E$-module spectrum}    \}$,  then this defines a functor $\hat{L}_\mathcal{G}$ on $\HoS$ which generalizes the completion functors studied in \cite{BK2}, \cite{BH},  and \cite{BT}.  We usually just write $\hat{L}_{\mathcal{G}}A$ as $\hat{A}_E$ in this case. \break

Definition 5.1 Chapter 1 of \cite{BK2} concerning $p$-completion generalizes to:

\begin{definition}[See Definition 8.3 of \cite{BO9}]   
\label{K-good-bad}
A space $A\in\mathcal{S}_*$ is called 
\begin{enumerate}
\item $E$-complete  if $\alpha:A \xrightarrow{\simeq} \hat{A}_E$.
\item $E$-good  if $\hat{A}_E$ is $E$-complete.
\item $E$-bad  if $\hat{A}_E$ is not $E$-complete.
\end{enumerate}
 \end{definition}
 
For a suitable ring spectrum $E$, and a space $A$,  the $E$-completion of  $A$ is $E$-local,   and $A$ is $E$-good if and only if 
$A_E\simeq \hat{A}_E$.    This is the topic of Lemma \ref{E-completion-localization} in the appendix.

 In \cite{BT} we studied $\widehat{S}^{2n+1}_{K}$,  $K$-theory completion of $S^{2n+1}$, for $p$ odd and $n\ge 1$.  (This work was extended to $p=2$ and to various Lie Groups in \cite{BD} and elsewhere.)  
The combined results of \cite{BO1}, \cite{MT3}, and  \cite{BT}   
allowed us to conclude that the map 
$S^{2n+1}_{K}\to\widehat{S}^{2n+1}_{K}$, from the localization to the completion, induces an isomorphism in $\pi_i$ for $i\ge 2$.  We know the fundamental group of $S^{2n+1}_{K}$ is trivial for $n\ge 1$ by Mislin \cite{Mis1}.  Our stated result in \cite{BT} included the case of  $\pi_1\widehat{S}^{2n+1}_{K}$  but in fact the methods there were not sufficient to determine that case.   
In this paper we show the following. 

\begin{theorem}
Let $n\ge 2$, $p$ is an odd prime.   If $(p-1) \vert (n-1)$, then  $\widehat{S}^{2n+1}_{K}$ is not simply connected and 
the sphere $S^{2n+1}$  is $K$-bad. 
\end{theorem}

The case of $\widehat{S}^{3}_{K}$ is still unsettled.  In all other cases, we can't say that the sphere is $K$-good in the strict sense of Definition \ref{K-good-bad} above because, while the $K$-localization is path connected, we don't know if the $K$-completion is.   However we can say that the map from the $K$-localization to the component of the $K$-completion that contains the basepoint is a homotopy equivalence. 

  We are making use of the important and beautiful results of Bousfield in \cite{BO6}.  We hope that further applications of Bousfield's ideas result from this paper.

\section{\bf The $p$-adic $K$-completion of a fiber square}

Expanding a little on the previous section, we briefly recall some definitions from \cite{BO9}.  We have a left proper pointed simplicial model category $\mathcal{C}$ and we're given a class of group objects 
$\mathcal{G} \subset \Ho\mathcal{C}$.   
A map 
$i:A\to B$
 in $\Ho\mathcal{C}$ is called $\mathcal{G}$-{\it monic} when $i^*:[B,G]_n \to [A,G]_n$ is onto for each $G\in \mathcal{G}$ and $n\ge 0$.   An object $Y\in\Ho\mathcal{C}$ is called $\mathcal{G}$-{\it injective} if $i^*:[B,Y]_n \to [A,Y]_n$ is onto for every $\mathcal{G}$-monic map $i:A\to B$ and $n\ge 0$.   We assume every object is the source of a $\mathcal{G}$-monic map to a $\mathcal{G}$-injective target.    The objects in $\mathcal{G}$ are called $\mathcal{G}$-{\it injective models}, and $\Ho\mathcal{C}$ is said to {\it have enough $\mathcal{G}$-injectives}. 

Bousfield proceeds to construct a model category structure on $c\mathcal{C}$, denoted by $c\mathcal{C}^{\mathcal{G}}$,  which is a generalization of the construction in \cite{DKS}.  The weak equivalences, called $\mathcal{G}$-equivalences,  are maps $f:X^{\bullet}\to Y^{\bullet}$ such that  
\[   f^*: [Y^{\bullet},G]_n \to [X^{\bullet},G]_n \]
is a weak equivalence of simplicial groups for all $G\in\mathcal{G}$, $n\ge 0$.
Regarding an object $A\in\mathcal{C}$ as the constant cosimplicial object in $c\mathcal{A}$, a resolution of $A$ is a map
$A\to Y^{\bullet}$ which is a trivial cofibration  to a fibrant object in the model category $c\mathcal{C}^{\mathcal{G}}$.   A resolution is used to define derived functors and completions.  The $\mathcal{G}$-completion of $A$ is defined to be $\Tot(Y^{\bullet})$. 

For objects in $\mathcal{C}$ a map $i:A\to B$  is called a $\mathcal{G}$-{\it equivalence} if it is a $\mathcal{G}$-equivalence regarding $A$ and $B$ as constant cosimplicial objects.  Bousfield proves that a $\mathcal{G}$-equivalence gives a homotopy equivalence of completions (Corollary 6.7 \cite{BO9}). 
He also shows that for purposes of computing derived functors and constructing completions, it is sufficient to have a {\it weak resolution}:  this is merely a 
$\mathcal{G}$-equivalence $A\to Y^{\bullet}$  such that each $Y^n$ is $\mathcal{G}$-injective.

Now we specialize to the cases where $\mathcal{C}$ is $\S$ or $\Sp$, and completion will be  Bousfield's $p$-complete version of $K$-completion.  
For a space or a spectrum $Z$ let $Z^{\wedge}$ denote the $p$-completion as in \cite{BO1}, \cite{BO2} and \cite{BK2}.  Let 
\begin{equation}\label{stable-K-theory-models}
{\mathcal{H}} = 
\{  N \vert\,\,\text{ $N$ is a $K$-module spectrum} \}\subset  \HoSp,  
\end{equation}
\begin{equation}\label{stable-p-complete-K-theory-models}
\hat{\mathcal{H}} = 
\{  N \vert\,\,\text{ $N$ is a $p$-complete $K$-module spectrum} \}\subset  \HoSp, 
\end{equation}
and
\begin{equation}\label{unstable-p-adic-K-completion-functor}
\hat{\mathcal{G}} = 
\{  \Omega^{\infty}N \vert\,\,\text{ $N$ is a $p$-complete $K$-module spectrum} \}\subset  \HoS
\end{equation}
all of which are classes of injective models.
 The resulting $\hat{\mathcal{G}}$-completion functor is what Bousfield calls the $p$-adic $K$-completion of a space $A$, denoted $\hat{A}_{\hat{K}}$.  
 (The class $\hat{\mathcal{G}}$ is denoted $\hat{\mathcal{G}^{'}}$ in \cite{BO9}.) The same definition can be applied to spectra 
 using $\hat{\mathcal{H}}$.  Bousfield applies this completion to a fiber square. 
 
 \begin{theorem}[Theorem 11.7 \cite{BO9}]\label{completion-of-a-fiber-square} Suppose we have a fiber square in $\S$
 \begin{equation}\label{fiber-square}
 \begin{tikzcd}
C \arrow{d}\arrow{r} & B\arrow{d} \\
A \arrow{r}  &\Lambda \\
 \end{tikzcd}
\end{equation}
in which all the spaces have torsion-free $K^{*}(-; \hat{\Z}_p)$-cohomology.   Suppose further that the $K_{*}(-; {\Z}/p)$-cobar spectral sequence collapes strongly, which means 
\[
\Cotor_{s}^{K_*(\Lambda;\Z/p)}(K_*(A;\Z/p),K_*(B;\Z/p)) = 
\begin{cases} 
K_*(C;\Z/p)\,\,\,\text{for $s=0$}\\
0 \,\,\, \text{otherwise}.\end{cases}
\]
Then the $p$-adic $K$-completion functor carries (\ref{fiber-square}) to a homotopy fiber square. 
\end{theorem}
 
Our example of this is
\begin{theorem}
The fiber sequence 
\[
F \to QS^{2n+1}\xrightarrow{j_p} Q(\Sigma^{2n+1}B_{q(n+1)-1}),
\]
where $j_p$ is the $p^{\text{th}}$ Hopf-James map and $B_{q(n+1)-1}$ is a stunted $B\Sigma_p$ localized at $p$ (see \cite{MT3}),
satisfies the hypothesis of Theorem \ref{completion-of-a-fiber-square}.  Here $Q$ denotes the functor $\Omega^{\infty}\Sigma^{\infty}$. 
\end{theorem}

\begin{proof}   Let $i:S^{2n+1}\to F$ be the lifting of the bottom cell.   For the first part of the hypothesis, 
the fact that the $K^{*}(-; \hat{\Z}_p)$-cohomology of the base space and total space are torsion free follows from Theorem 8.3  of \cite{BO6}.   In \cite{MT3}
it is proven that $i_*:K_*(S^{2n+1}) \to K_*(F)$ is an isomorphism so $K^{*}(F; \hat{\Z}_p)$ is torsion free as well.

For the second part loop the fiber sequence back and apply the  $K_{*}(-; {\Z}/p)$-bar spectral sequence to

\[
QS^{2n} \xrightarrow{\Omega j_p}  Q(\Sigma^{2n}B_{q(n+1)-1})  \to F 
\]

From \cite{MC} we know that as modules over $K_*(\text{pt};\Z/p)$
\begin{align*}
K_*(QS^{2n+1};\Z/p) = E(x_1,x_2,\dots)  \\ 
K_*(Q(\Sigma^{2n+1}B_{q(n+1)-1});\Z/p) = E(y_1,y_2,\dots),\\
K_*(QS^{2n};\Z/p) = P(u_1,u_2,\dots),\\
K_*(Q(\Sigma^{2n}B_{q(n+1)-1});\Z/p) = P(v_1,v_2,\dots).\\
\end{align*}
 Here $E$ denotes an exterior algebra, $P$ a polynomial algebra. 

In \cite{MT3} Lemma 2.2 it is proved that $(\Omega j_p)_*(u_1) = 0$ and $(\Omega j_p)_*$ maps the algebra $P(u_2,u_3,\dots)$ isomorphically onto  $P(v_1,v_2,\dots)$.   (Note that this fact and the bar spectral sequence imply that $i:S^{2n+1}\to F$ induces an isomorphism in $K$-theory. )

There is the diagram where the horizontal sequences are path-loop fiber sequences

\begin{equation}
\begin{tikzcd}
QS^{2n} \arrow[d,"\Omega j_p"]\arrow{r} & * \arrow{r}\arrow{d} & QS^{2n+1}\arrow[d,"j_p"]\\
Q(\Sigma^{2n}B_{q(n+1)-1}) \arrow{r}  &*\arrow{r} & Q(\Sigma^{2n+1}B_{q(n+1)-1})\\
 \end{tikzcd}
\end{equation}

To compute $\Tor^{K_*(QS^{2n};\Z/p)}_{s}(K_*(\text{pt};\Z/p),\Z/p)$ use the resolution \break $E(x_1,x_2,\dots)\otimes P(u_1,u_2,\dots)$.   The bar spectral sequence for the top row collapses. 
For $\Tor^{K_*(Q(\Sigma^{2n}B_{q(n+1)-1});\Z/p)}_{s}(K_*(\text{pt};\Z/p),\Z/p)$ use the resolution $E(y_1,y_2,\dots)\otimes P(v_1,v_2,\dots)$ and the bottom row bar spectral collapses.   The conclusion is that 
\[(j_p)_*: K_*(QS^{2n+1};\Z/p) \to K_*(Q(\Sigma^{2n+1}B_{q(n+1)-1});\Z/p)\] 
is a surjective map of coalgebras.     As mentioned in \cite{BO9}, this implies by Theorem 10.11 of  \cite{BO6} that the hypothesis of a strongly collapsing $K_*(-,\Z/p)$-theory cobar spectral sequence is satisfied. 

\end{proof}

The upshot of all of this is that there is a fiber sequence
\begin{equation}
\widehat{F}_{\hat{K}} \to \widehat{QS^{2n+1}}_{\hat{K}}   \xrightarrow{\widehat{j_p}} \widehat{Q(\Sigma^{2n+1}B_{q(n+1)-1})}_{\hat{K}}.
\end{equation}
Since $S^{2n+1} \to F$ induces an isomorphism $K$-homology,  it induces an isomorphism in  $N$-cohomology for every $p$-complete $K$-module spectrum $N$ by Lemma 13.1 of \cite{AD4}. This means that it is a $p$-adic $K$-equivalence, therefore    
$\widehat{S^{2n+1}}_{\hat{K}} \to \hat{F}_{\hat{K}}$  is an equivalence and the above fiber sequence becomes
\begin{equation}\label{main-long-exact-sequence} 
\widehat{S^{2n+1}}_{\hat{K}} \to \widehat{QS^{2n+1}}_{\hat{K}}   \xrightarrow{\widehat{j_p}} \widehat{Q(\Sigma^{2n+1}B_{q(n+1)-1})}_{\hat{K}}.
\end{equation}

%
%
\section{\bf The completion of some infinite loop spaces}

Next we want to compute the homotopy groups of $\widehat{QS^{2n+1}}_{\hat{K}}$  and  $\widehat{Q(\Sigma^{2n+1}B_{2(p-1)(n+1)-1})}_{\hat{K}}$.   In general $K$-completion doesn't commute with $\Omega^{\infty}$ but in the case of the spectra $\Sigma^{\infty}S^{2n+1}$ and $\Sigma^{\infty}B_{2(p-1)(n+1)-1}$ it does.

\begin{theorem}\label{main-completion-theorem}
Let  $X$ be either of the spaces  $S^{2n+1}$ or \break $B_{q(n+1)-1}$.  The subscript zero denotes the $0$-connected cover. 
Then
\begin{equation}
(\widehat{QX}_{\hat{K}})_{0} = 
(\widehat{\Omega^{\infty}\Sigma^{\infty}X}_{\hat{K}})_{0} \cong \Omega^{\infty}_{0}(\widehat{\Sigma^{\infty}X}_{\hat{K}}) 
\end{equation} 
\end{theorem}

Before proving this we start by saying a little more about 
the $p$-adic $K$-completion functor, and produce an unstable resolution by looping down a stable resolution. 

Let $\Sp$ denote the model category of spectra as in \cite{HSS}, and let $\HoSp = \Ho \Sp$ denote the stable homotopy category.  
We have the classes of injective models ${\mathcal{H}}$ and $\hat{\mathcal{H}}$ defined in (\ref{stable-K-theory-models}) and 
(\ref{stable-p-complete-K-theory-models}).   
Since $K$ is represented by a structured ring spectrum in $\Sp$ (see \cite{MQRT}), there is a 
 triple on $\Sp$ which takes a spectrum $A$ to $K\wedge A$, with the unit induced by $S\to K$.
This triple satisfies all the conditions of Theorem 7.4 in \cite{BO9} with respect to ${\mathcal{H}}$, and  therefore the triple resolution $A\to K^{\bullet}A$ is a weak ${\mathcal{H}}$-resolution of $A$ and can be used to compute the homotopy groups of $\hat{A}_{{K}}$.  Stably $K$-theory completion is well understood. 
Now $p$-complete the spectra $K^{{\wedge}n}A$ to get a map of cosimplicial spectra $A\to \widehat{K^{\bullet}A}$ and apply the same argument given in the proof of 11.5 of \cite{BO9}:  the map $A\to K^{\bullet}A$ is an ${\mathcal{H}}$-equivalence, hence a  $\hat{\mathcal{H}}$-equivalence since $\hat{\mathcal{H}} \subset {\mathcal{H}}$.  The map $K^{\bullet}A \to \widehat{K^{\bullet}A}$ is obviously an $\hat{\mathcal{H}}$-equivalence.   
The spectra $\widehat{K^{n}A}$ are 
$\hat{\mathcal{H}}$-injective, so $A\to \widehat{K^{\bullet}A}$ is a weak $\hat{\mathcal{H}}$ resolution.  Replace this with a Reedy fibrant replacement
 which, by a harmless abuse of notation, we still refer to as $A\to {\widehat{K^{\bullet}A}}$.

\begin{proof}[Proof of Theorem \ref{main-completion-theorem}]
For either of  the spaces $X$ of Theorem \ref{main-completion-theorem}, apply $\Sigma^{\infty}$ and the construction given above in the stable category
\begin{equation} \label{stable-resolution}
\Sigma^{\infty}X \to \widehat{K^{\bullet}  \Sigma^{\infty}X}.
\end{equation}
Then apply the functor $\Omega^{\infty}$ to this 
to obtain an augmented cosimplicial space
\begin{equation}
\Omega^{\infty}\Sigma^{\infty}X \to \Omega^{\infty}(\widehat{K^{\bullet}\Sigma^{\infty}X}). 
\end{equation} 
This cosimplicial space is level wise $\hat{\mathcal{G}}$-injective, since $\Omega^{\infty}$ of a $\widehat{\mathcal{H}}$-injective spectrum is 
a $\widehat{\mathcal{G}}$-injective space. 

Taking the component containing the basepoint,
\begin{equation}
\label{unstable-resolution}
\Omega^{\infty}\Sigma^{\infty}X \to \Omega^{\infty}_{0}(\widehat{K^{\bullet}\Sigma^{\infty}X})
\end{equation}
is an augmented cosimplicial space, and the target
is also  termwise $\widehat{\mathcal{G}}$-injective since the $n$th space is a retract
of the $n$th space of $\Omega^{\infty}(\widehat{K^{\bullet}\Sigma^{\infty}X})$.  

Using Theorem \ref{looping-down-the-triple-resolution} below, which will be proved in the next section, we conclude 

\[\widehat{\Omega^{\infty}\Sigma^{\infty}X}_{\hat{K}} = \Tot(\Omega^{\infty}_{0}(\widehat{K^{\bullet}\Sigma^{\infty}X})).\]

From the spectral sequence of a cosimplicial space we get 
\[\pi_n\Tot(\Omega^{\infty}(\widehat{K^{\bullet}  \Sigma^{\infty}X})) \cong 
\pi_n\Tot(\Omega^{\infty}_{0}(\widehat{K^{\bullet}  \Sigma^{\infty}X})),\quad\quad n\ge 1.\]

\vspace{.25cm}
By Lemma \ref{Tot-and-Omega-infinity-commute} we have 
\[
\Tot(\Omega^{\infty}(\widehat{K^{\bullet}  \Sigma^{\infty}X}))  =   \Omega^{\infty}\Tot(\widehat{K^{\bullet}  \Sigma^{\infty}X}).
\]
and the conclusion follows.

\end{proof}

\begin{theorem}\label{looping-down-the-triple-resolution}
Let  $X$ be as in Theorem \ref{main-completion-theorem}.  The map  (\ref{unstable-resolution})
of cosimplicial spaces is a weak $\hat{\mathcal{G}}$-resolution.
\end{theorem}

The proof amounts to showing that
\begin{equation}
N^{*}(\Omega^{\infty}\Sigma^{\infty}X)  \xleftarrow{} N^{*}(\Omega^{\infty}_{0}(\widehat{K^{\bullet}\Sigma^{\infty}X}))
\end{equation}
is a weak equivalence of simplicial abelian groups for every $p$-complete $K$-module spectrum $N$.

\section{\bf Bousfield's computation of $K^{*}(\Omega_{0}^{\infty}E;\Z_p)$}

In order to proceed we must summarize some of the main results of \cite{BO6}. In that paper Bousfield proves the remarkable theorem that if $E$ is a spectrum with  
$K^*(E;\Z_p)$ torsion-free,  then $K^*(\Omega_{0}^{\infty}E;\Z_p)$ is an algebraically 
defined functor of $K^*(E;\Z_p)$.  The $K$-cohomology of the infinite loop space is regarded as an object in the category of $Z/2$-graded $p$-adic $\Lambda$-rings  and the 
$K$-cohomology of the spectrum is an object in the category of $Z/2$-graded stable $p$-adic Adams modules.   
(Bousfield introduced this terminology in \cite{BO7}. In \cite{BO6} he calls them $Z/2$-graded $p$-adic $\psi^{*}$-modules.)
This section consists of a brief outline of the (somewhat elaborate) relevant definitions.   See \cite{BO6} for proofs.

\begin{definition} A $\Z/2$-graded finite stable $p$-adic Adams module is  a $\Z/2$-graded finite Abelian $p$-group $F$ with endomorphisms $\psi^{k}:F\to F$ for 
$k\in \Z - p\Z$ such that:
\renewcommand{\theenumi}{\roman{enumi}}
\begin{enumerate}
\item $\psi^{1} = $ {Id} and $\psi^{j}\psi^{k} = \psi^{jk}$ for all $j,k \in \Z-p\Z$.
\item There exists an integer $n\ge 1$ such that $\psi^k = \psi^{k+p^nj}$ on $F$ for all $k\in \Z - p\Z$ and $j\in \Z$.
\end{enumerate}
A $\Z/2$-graded stable $p$-adic Adams module is  the topologized inverse limit of a system of  $\Z/2$-graded finite stable $p$-adic Adams modules.
(In \cite{BO6} this is called a $\Z/2$-graded $p$-adic $\psi^{*}$-module.)

A $\Z/2$-graded stable $p$-adic Adams module $M$ is called linear if $\psi^{k} = k$ on $M^{0}$ and $\psi^{k} = 1$ on $M^{1}$, for $k\in \Z - p\Z$. 
\end{definition}

As noted in \cite{BO6},  if $E$ is an arbitrary spectrum then $\{K^{0}(E;\Z_p),K^{1}(E;\Z_p)\}$ is a $\Z/2$-graded stable $p$-adic Adams module with the usual Adams operations.

On the unstable side start with the following:

\begin{definition}
A  $\theta^{p}$-ring $R$ is  a commutative ring with identity together with a function $\theta^{p}:R\to R$ satisfying:
\renewcommand{\theenumi}{\roman{enumi}}
\begin{enumerate}
\item $\displaystyle \theta^{p}(a+b) = \theta^{p}(a) + \theta^{p}(b) - \sum_{i=1}^{p-1}\limits\dfrac{1}{p}\binom{p}{i}a^ib^{p-i}$.
\item $\theta^{p}(ab) = (\theta^{p}a)b^{p}+ a^{p}(\theta^{p}b) + p(\theta^{p}a)(\theta^{p}b)$.
\item $\theta^{p}(1) =0$.
\end{enumerate}
A $\psi^{p}$-module $M$ is just an abelian group with an endomorphism \break$\psi^{p}:M\to M$.
We can make a $\theta^{p}$-ring $R$  into a $\psi^{p}$-module by defining 
$\psi^{p}:R\to R$ by $\psi^{p}(a) = a^{p} + p\theta^{p}(a)$.   Then we have 
$\psi^{p}\theta^{p} = \theta^{p}\psi^{p}$. 
\end{definition}

\begin{definition} 
A $\Z/2$-graded $\theta^{p}$-ring $R =\{R^{0},R^{1}\}$ is a strictly commutative $\Z/2$-graded ring for which $R^{0}$ is a $\theta^{p}$-ring and $R^{1}$ is a $\psi^{p}$-module,  and these actions satisfy:
\renewcommand{\theenumi}{\roman{enumi}}
\begin{enumerate}
\item $\psi^{p}(ax) = \psi^{p}(a)\psi^{p}(x)$ for $a\in R^{0}$ and $x\in R^{1}$.
\item $\theta^{p}(xy) = \psi^{p}(x)\psi^{p}(y)$ for $x,y \in R^{1}$. 
\end{enumerate}
\end{definition}

We can make $\Z_p$ into a  $\theta^{p}$-ring by defining $\theta^{p}(a) = (a-a^p)/{p}$.  Then $\psi^{p}(a) = a$.   We say a $\theta^{p}$-ring $R$ is a $\theta^{p}$-ring over $\Z_p$ if there is a map of $\theta^{p}$-rings $\Z_p\to R$.

\begin{definition}
A $\Z/2$-graded finite $p$-adic $\theta^{p}$-ring $S$ is an augmented $\Z/2$-graded $\theta^{p}$-ring over $\Z_p$ such that:
\renewcommand{\theenumi}{\roman{enumi}}
\begin{enumerate}
\item the augmentation ideal $\tilde{S}$ is finite $p$-torsion and nilpotent.
\item for each $x\in \tilde{S}^{0}$ and $y\in {S}^{1}$ there is an $n>0$ with \break$(\theta^{p})^{n}(\theta^{p}x-x) = 0$ and 
$(\psi^{p})^{n}(\psi^{p}y - y) = 0$. 
\end{enumerate}
\end{definition}

\begin{definition}
A $\Z/2$-graded finite $p$-adic $\theta^{p}$-ring $S$ is said to be equipped with $p$-adic Adams operations if it is has endomorphisms 
\break$\psi^{k}:S\to S$ for $k$ prime to $p$ which satisfy:
\renewcommand{\theenumi}{\roman{enumi}}
\begin{enumerate}
\item $\theta^{p}\psi^{k} = \psi^{k}\theta^{p}$, $\psi^{p}\psi^{k} = \psi^{k}\psi^{p}$, $\psi^{1} = 1$, and $\psi^{j}\psi^{k} = \psi^{jk}$.
\item the operations $\psi^{k}:\tilde{S} \to \tilde{S}$ are periodic in $k\in Z^{+} -  p\Z_p$ with some period $p^r$. 
\item $\psi^{k}x \cong kx \mod \Gamma^{2}\tilde{S}^{0}$ and $\psi^{k}y \cong y \mod \Gamma^{2}S^{1}$ for each $k\in \Z- p\Z$, $x\in \tilde{S}^{0}$,
and $y\in S^{1}$,  where 
\begin{align*}
 \Gamma^{2}\tilde{S}^{0} &= \{x \in \tilde{S}^{0} \vert (\theta^{p})^{n} x = 0\, \text{for some $n> 0$}\},\\
  \Gamma^{2}{S}^{1} &= \{x \in {S}^{1} \vert (\psi^{p})^{n} x = 0\, \text{for some $n> 0$}\}.
\end{align*}
\end{enumerate}
A $\Z/2$-graded $p$-adic $\theta^{p}$-ring equipped with $p$-adic Adams operations  is the topologized inverse limit of an inverse system of 
$\Z/2$-graded finite $p$-adic $\theta^{p}$-rings equipped with $p$-adic Adams operations.
\end{definition} 

Finally we can say that it is shown in \cite{BO6} that if $X$ is a connected CW-complex, then $K^{*}(X;\Z_p)$ is A $\Z/2$-graded $p$-adic $\theta^{p}$-ring equipped with $p$-adic Adams operations.
  
\vspace{.2cm}
  
\noindent {\bf Note:}   Bousfield spends much of \cite{BO6} studying $\Z/2$-graded $p$-adic $\Lambda$-rings.    He proves that this notion is the same thing as a 
$\Z/2$-graded $p$-adic $\theta^{p}$-ring equipped with $p$-adic Adams operations.  In this paper we will make use of the $\theta^{p}$-ring description of the $p$-adic $K$-cohomology of a space rather than the $\Lambda$-ring description.   

\subsection{The functor $W$}

Now we  summarize Bousfield's definition of the functor which determines the $p$-adic $K$-cohomology of $\Omega^{\infty}E$ in terms of that of $E$.   We use the following notation.

%
%
\begin{itemize}
\item $\abp$ =  the category of $\Z/2$-graded $p$-profinite Abelian groups.
\item $\babp$ = the category of morphisms between two objects in $\abp$.
\item $\thetapring$  = the category of $\Z/2$-graded $p$-adic $\theta^{p}$-rings.
\item $\Adp$ = the category of $\Z/2$-graded stable $p$-adic Adams modules.
\item $\bar{\Adp}$ = the category of maps between two objects in $\Adp$.
\item $\bar{\Adp^{\text{lin}}}$ = the subcategory of  $\bar{\Adp}$ with linear target.
\item $\thetapringAd$  = the category of $\Z/2$-graded $p$-adic $\theta^{p}$-rings equipped with $p$-adic Adams operations.
\end{itemize}

\begin{definition}[7.4 of \cite{BO6}]
The functor $W:\babp \to \thetapring$ is the left adjoint of the forgetful functor, which takes $R$ to $(\tilde{R}\to \tilde{R}/\hat{\Gamma}^{2}\tilde{R})$. 
\end{definition}

The functor $W$ can be prolonged to a functor $W:   \bar{\Adp^{\text{lin}}}   \to    \thetapringAd      $ 
using the Adams operations on the source to define Adams operations on the target (8.2 of \cite{BO6}).
 
Suppose $E$ is a spectrum, and consider the connective cover $E\langle 0 \rangle \to E$, which induces an isomorphism in $K^{*}(-;\Z_p)$. This induces a map of $\Z/2$-graded stable $p$-adic Adams modules
\[K^{*}(E;\Z_p)_{H}:K^{*}(E;\Z_p) \to H^{\bullet}(E\langle 0 \rangle;\Z_p)\]
where the action of $\psi^{k}$ on $\{H^{2}(E(\langle 0 \rangle;\Z_p), H^{1}(E\langle 0 \rangle;\Z_p)\}$ is defined to be linear. 
Thus $K^{*}(E;\Z_p)_{H}$ is an object in $\bar{\Adp^{\text{lin}}}$.

The main topological theorem of \cite{BO6} is the following.

\begin{theorem}[Theorem 8.3 of \cite{BO6}, special case]\label{bousfields-theorem}
Suppose $E$ is a spectrum with $K^{*}(E;\Z_p)$ torsion free.   Then 
\[WK^{*}(E;\Z_p)_{H}  \cong   K^{*}(\Omega_{0}^{\infty}E;\Z_p)  \]
is an isomorphism in $\thetapringAd$.   Furthermore $K^{*}(\Omega_{0}^{\infty}E;\Z_p)$ is torsion free.
\end{theorem}


\vspace{.2cm}

\subsection{Proof of \ref{looping-down-the-triple-resolution}}

\begin{proof}
First we show that we have a weak equivalence of simplicial abelian groups in $p$-adic $K$-cohomology.

We have a map between simplicial objects in $\bar{\Adp}$, 
\begin{equation}
\label{stable-projective-resolution}K^{*}(X;\Z_p)_{H}  \xleftarrow{} K^{*}(K^{\bullet}X;\Z_p)_{H}
\end{equation}
with the target being a constant simplicial object.   Applying the forgetful functor we get a map in the category $s\babp$. 
The category $\abp$ is an abelian category with enough projectives, as is the diagram category $\babp$.    We give $\ch\babp$, the category of non-negatively graded chain complexes, the projective model category structure. Then we transport this to a model category structure on $s\babp$
via the Dold-Kan equivalence
\[  K:   \ch\babp \rightleftarrows  s\babp :N .\]
See \cite{GS} for the details of these constructions. 

The map (\ref{stable-projective-resolution}) is weak equivalence between fibrant objects because that's true in the category of maps between simplicial abelian groups.   The source and target of  (\ref{stable-projective-resolution}) are cofibrant by Lemma \ref{cofibrancy}.  It follows from Whitehead's Theorem (for example see Theorem 1.10 page 73 of \cite{GJ}) that (\ref{stable-projective-resolution}) is a homotopy equivalence in $s\babp$ as well. 

Now apply the functor $W$ to get a homotopy equivalence in $s\thetapringAd$.   Apply the forgetful functor to simplicial abelian groups to get the desired the weak equivalence. 

To prove the result for $N$-cohomology,  where $N$ is an arbitrary $p$-complete $K$-module spectrum, observe that since $N$ is $p$-complete, the coefficients $N_{*}$ are Ext-$p$-complete abelian groups,  therefore by Lemma 11.11 of \cite{BO9},  the $N$-cohomology depends functorially on the $p$-adic $K$-cohomology, and the result follows. 

\end{proof}

%
%

\section{\bf Conclusion}

For our spectra  $\Sigma^{\infty}X$ the stable $p$-adic $K$-completion  is just the $p$-completion of the $K$-theory localization \cite{BO2}.  This is because there is a homotopy fiber square obtained by applying Proposition 2.9 of \cite{BO2} termwise to the cosimplicial spectra:

\begin{equation}
 \begin{tikzcd}
\Tot(K^\bullet\Sigma^{\infty}X)  \arrow{d}\arrow{r} & \Tot(\widehat{K^\bullet\Sigma^{\infty}X})\arrow{d} \\
\Tot(K^\bullet\Sigma^{\infty}X)_{(0)} \arrow{r}  &\Tot(\widehat{K^\bullet\Sigma^{\infty}X})_{(0)} \\
 \end{tikzcd}
\end{equation}
Since the bottom two spectra are $HQ$-local so is the fiber of the bottom map, so the same is true of the fiber of the top map. This means the top map induces an isomorphism in mod-$p$ homotopy groups, which implies that it is $p$-completion.  

The homotopy groups of the $K$-theory localizations of the spectra  $\Sigma^{\infty}S^{2n+1}$ and  $\Sigma^{\infty}B_{2(p-1)(n+1)-1}$ are well known (see for example \cite{TH1} or \cite{MT3}).  In the long exact sequence of homotopy groups for (\ref{main-long-exact-sequence}), when  $(p-1) \vert (n-1)$,  the group $\pi_{2}(\widehat{Q(\Sigma^{2n+1}B_{2(p-1)(n+1)-1})}_{\hat{K}})$ is non-trivial and the boundary homomorphism
\[
\pi_{2}(\widehat{Q(\Sigma^{2n+1}B_{2(p-1)(n+1)-1})}_{\hat{K}}) \xrightarrow{\partial} \pi_{1}(\widehat{S^{2n+1}}_{\hat{K}}).
\] 
is non-zero. This follows from the calculations of \cite{BT}, \cite{TH1} and \cite{MT3}.
Since the completion is not simply connected it is not equivalent to the localization.

\section{\bf Appendix}

%
%

\begin{lemma}\label{E-completion-localization} 
Suppose $E$ is a ring spectrum in the stable homotopy category. Then for a space or spectrum $A$, $\hat{A}_{E}$ is $E$-local.   If we further assume that $E_{*}$-acyclic spectra are the same as $E^*$-acyclic spectra then    
$A$ is $E$-good if and only if the natural map $A_{E}\xrightarrow{\cong }\hat{A}_{E}$ (between the Bousfield localization and the completion)
is a homotopy equivalence. 
\end{lemma}

\begin{proof}  Take $\mathcal{G}$  to be the class $\mathcal{G} = \{\Omega^{\infty}N \vert\,\,\text{ $N$ 
is an $E$-module spectrum} \}$ 

Every $\mathcal{G}$-injective space or spectrum $A$ is $E$-local:  suppose $C$ satisfies $E_*(C) =0$.   Then for any $E$-module spectrum $N$, we have $N^*(C) = 0$ ( Lemma 13.1 of \cite{AD4}).  This means $C\to *$ is $\mathcal{G}$-monic.   Since $A$ is $\mathcal{G}$-injective,  $[C,A]_{*} = 0$ so $A$ is $E$-local.

If $A\to Y^{\bullet}$ is an $\mathcal{G}$-resolution, each $Y^{n}$ is $\mathcal{G}$-injective, hence  $E$-local, so $\hat{A}_E = \Tot Y^{\bullet} \cong \holim Y^{\bullet}$ is $E$-local.
This gives a diagram 
 \begin{equation}
 \begin{tikzcd}
A \arrow{rd} \arrow{r}& A_{E}\arrow{d} \\
{} & \hat{A}_{E} \\
 \end{tikzcd}
\end{equation}

Suppose the natural map $A_{E}\xrightarrow{}\hat{A}_{E}$ is an equivalence.  Then $A\to \hat{A}_E$ is an $E_*$-isomorphism,  which implies it's an 
$N^*(\hspace{1mm})$-isomorphism for any $E$-module spectrum $N$, which means it's an $\mathcal{G}$-equivalence.    By Proposition 8.5 of \cite{BO9},  $A$ is ${E}$-good.

Conversely if $A$ is $E$-good then by Proposition 8.5 of \cite{BO9} the map \break$A\to \hat{A}_{E}$ is an $E$-equivalence.    That implies that it is an $E^{*}$-isomorphism and by hypothesis an $E_*$-isomorphism.  Since $\hat{A}_{E}$ is $E$-local this implies an equivalence $A_{E}\xrightarrow{\cong }\hat{A}_{E}$.
\end{proof}

Examples of spectra satisfying the hypotheses of Lemma \ref{E-completion-localization} are Johnson-Wilson spectra $E(n)$ (\cite{HOV4}, \cite{DEV2}).

%
%

\begin{lemma}\label{Tot-and-Omega-infinity-commute}
Consider the Quillen adjunction
\[\Sigma^{\infty}:  \S \rightleftarrows \Sp: \Omega^{\infty}\]
between the category of pointed spaces and the category of spectra.   
Then for every Reedy fibrant $X^{\bullet}\in c\Sp$,   $\Omega^{\infty}(X^{\bullet})$  is Reedy fibrant in $c\S$ and 
$\Omega^{\infty}\Tot(X^{\bullet}) = \Tot(\Omega^{\infty}(X^{\bullet})) $ . 
\end{lemma}

\begin{proof} For the first statement, a cosimplicial object $X^{\bullet}$ is Reedy fibrant when  $X^{n}\to M^nX^{\bullet}$ is a fibration for each $n$, where
the $n$th matching object $M^{n}X^{\bullet}$ is defined as a certain limit.
 Since $\Omega^{\infty}$ is a right Quillen adjoint it preserves limits and fibrations, so if $X^{\bullet}$ is a Reedy fibrant cosimplicial spectrum, then $\Omega^{\infty}(X^{\bullet})$ is a Reedy fibrant cosimplicial space. 

For the second statement, we have that since $X^{\bullet}$ is Reedy fibrant, $\Tot(X^{\bullet})$ is the limit of the diagram (e.g. Chapter VIII, Section 1 of \cite{GJ} )

\[
\begin{tikzcd}
\displaystyle\prod_{n\ge 0} (X^n)^{\Delta^{n}}    \arrow[r, shift left=.5ex]\arrow[r, shift right=1.0ex]
& \displaystyle\prod_{\phi:\bold{n}\to\bold{m}} (X^m)^{\Delta^{n}}, 
\end{tikzcd}
\]
i.e. the equalizer.   Since $\Omega^{\infty}$ is a right adjoint, and hence preserves limits, we have
$\Omega^{\infty}\Tot(X^{\bullet})$ is the limit of 
\[
\begin{tikzcd}
\displaystyle\prod_{n\ge 0} \Omega^{\infty}((X^n)^{\Delta^{n}}) \arrow[r, shift left=.5ex]  \arrow[r, shift right=1.0ex] & \displaystyle\prod_{\phi:\bold{n}\to\bold{m}} \Omega^{\infty}((X^m)^{\Delta^{n}}). 
\end{tikzcd} 
\]
By definition of the cotensor objects in $\S$ and $\Sp$ (e.g. Chapter II, Section 2 of \cite{GJ} and Definition 1.2.9 of \cite{HSS}) we have $\Omega^{\infty}(X^{\Delta^{n}}) = (\Omega^{\infty}X)^{\Delta^{n}} $, 
so this limit is also the limit of 
\[
\begin{tikzcd}
\displaystyle\prod_{n\ge 0} (\Omega^{\infty}X^n)^{\Delta^{n}} \arrow[r, shift left=.5ex]  \arrow[r, shift right=1.0ex]& \displaystyle\prod_{\phi:\bold{n}\to\bold{m}} (\Omega^{\infty}X^m)^{\Delta^{n}}. 
\end{tikzcd}
\] 
Since $\Omega^{\infty}(X^{\bullet})$ is Reedy fibrant, this limit is $\Tot(\Omega^{\infty}(X^{\bullet}))$.

\end{proof}

%
%

\begin{lemma}\label{cofibrancy}
Let $X$ be as in Theorem \ref{looping-down-the-triple-resolution}.   The simplicial objects $K^{*}(X;\Z_p)_{H}$ and $K^{*}(K^{\bullet}X;\Z_p)_{H}$ are cofibrant in $s\babp$.
\end{lemma}

\begin{proof}
Call either simplicial object $Y_{\bullet}$.
 In Proposition 4.2 of \cite{GS} (we're applying this to an arbitrary Abelian category) the authors give an explicit description of a cofibrant object.
Note that all the groups in $Y_{\bullet}$ are torsion free by Theorem \ref{bousfields-theorem}, so they are projective.   Also in each simplicial degree 
the map is a surjective map and so it is a projective object in $\babp$.  Apply the composite $KN$.  $NY_{\bullet}$ is a chain complex consisting of torsion free groups, so projective groups.   By condition (3) of Proposition 4.2 of \cite{GS} $KNY_{\bullet}$ is cofibrant in 
$s\babp$ and $Y_{\bullet}$ is isomorphic to $KNY_{\bullet}$ by Dold-Kan.   
\end{proof}

\bibliography{rob}{}
\bibliographystyle{plain}

\end{document}